\documentclass[12pt]{amsart}

\newtheorem{Main}{Theorem}

\usepackage{mathabx}

 \usepackage{ulem}
 \def\hf{\hat{f}}
 \def\cS{\mathcal S}
 \def\simmple{\cS}
 \def\bbbbb{}
 \def\tvarphi{\widetilde{\varphi}}
 \usepackage{txfonts}

\usepackage{graphicx,psfrag,epsfig}
\usepackage{amsmath}
\usepackage{amsthm}
\usepackage[a4paper]{geometry}
\usepackage{amsfonts}
\usepackage{amssymb}
\usepackage{tikz}
\usetikzlibrary{arrows,shapes,trees}
\usepackage{graphicx}
\usepackage{amssymb}
\usepackage{amsfonts}
\usepackage{amsthm}
\usepackage{amsmath}
\usepackage{mathrsfs}

\def\epsilon{\eps}

\UseRawInputEncoding

\def \cM1{{\mathcal M_1}}

\def\cM{{\mathcal{M}}}

\def\<{{\left<}}
\def\>{{\right>}}

\def\hf{\hat{f}}

\usepackage{mathrsfs}
\usepackage{color}

\definecolor{dgreen}{rgb}{0.1,0.6,0.1}

\definecolor{bluegreen}{rgb}{0.1,0.5,0.2}
\definecolor{bpurple}{rgb}{0.74,0.2,0.64}
\def\bp{\bpurple}

\def\black{\color{black}}

\def\bpurple{\color{bpurple}}

\definecolor{orange}{rgb}{0.8, 0.33, 0.0}

\def\cM{\mathcal M}

\definecolor{orange}{rgb}{1,0.5,0}

\definecolor{orange}{rgb}{1,0.5,0}

\definecolor{bluegray}{rgb}{0.6, 0.6, 0.8}

\definecolor{bluered}{rgb}{0.6,0.3,0.4}

\newtheorem{lemma}{Lemma}
\newtheorem{theorem}[lemma]{Theorem}
\newtheorem*{theorem*}{Theorem}

\newtheorem{proposition}[lemma]{Proposition}
\newtheorem*{proposition*}{Proposition}

\newtheorem{definition}{Definition}

\newtheorem*{Atheorem*}{Theorem}

\def\<{{\langle}}
\def\>{{\rangle}}

\def\a{{\alpha}}

\def\eps{{\varepsilon}}

\def\N{{\mathbb{N}}}
\def\R{{\mathbb{R}}}

\def\Z{{\mathbb{Z}}}
\def\T{{\mathbb{T}}}

\def \cM1{{\mathcal M_1}}

\def\cM{{\mathcal{M}}}

\def \tf{{\tilde f}}
\def \tg{{\tilde g}}

\def\<{{\left<}}
\def\>{{\right>}} 
\def\vphi{\varphi}

\begin{document}
\author{Fatna Abdedou, Bassam Fayad, Jean-Paul Thouvenot}
\title{Inducing countable Lebesgue spectrum}
\begin{abstract} We show that every ergodic dynamical system induces a system with pure Lebesgue spectrum of infinite multiplicity.
\end{abstract}

\maketitle
\section{Introduction} 
Given a dynamical system $(T,X,\mu)$, we define  for a measurable set $A\subset X$ with $\mu(A)>0$, the induced dynamical system $(T_A,A,\mu_A)$ with $T_A$ being the first return map to the set $A$, and for any measurable set $B\subset A$, $\mu_A(B)=\mu(B)/\mu(A)$.

Recall that, given a dynamical system $(T,X,\mu)$, the corresponding Koopman operator $U_T$ acts on the space $L^2_0(X,\mu)$  of complex square integrable zero  mean functions as $U_Tf=f\circ T$. As each unitary operator acting on a separable Hilbert space, $U_T$ is determined by its maximal spectral type $\sigma_T$ measure (a spectral measure on the circle $\T$ which dominates all other spectral measures) and by the multiplicity function $M_T:\T\to \{1,2,\ldots\}\cup\{\infty\}$   (defined $\sigma_T$-a.e.). 

We say $(T,X,\mu)$ has pure Lebesgue  spectrum when $\sigma_T$ is equivalent to the Lebesgue measure on the circle. 
We say $(T,X,\mu)$ has pure Lebesgue  spectrum with infinite multiplicity, when in addition $M_T(\cdot)=\infty$ Lebesgue almost surely. 

In \cite{delarue}, De La Rue showed that 
 any ergodic dynamical system $(T,X,\mu)$ 
 induces an ergodic dynamical system $(T_A,A,\mu_A)$ that has pure Lebesgue spectrum. He asked if it is possible to insure that the induced system has infinite Lebesgue spectrum. 
  In this note, we adapt the construction of De La Rue to show that the answer to this question is positive. 
  
\begin{Main} \label{main.infinite} 
For any ergodic dynamical system $(T,X,\mu)$ there exists an induced system that has a pure Lebesgue  spectrum with infinite multiplicity. 
\end{Main}

\section{Notations, Definitions and Preliminaries} 

\subsection{Weak and strong closeness between measures on the circle} 

\begin{definition} For $\a>0$ and $\tau>0$,  a function $\vphi \in C^0(\T,\R)$ is said to be $(\a,\tau)$-good if 
$\vphi(\cdot)\geq 0$ and $\vphi(\theta)>\a$ for $\theta$ outside $(-\tau,\tau)$.
\end{definition} 

\begin{definition}[Strong closeness of densities] Suppose $\vphi$ is $(\a,\tau)$-good. Given $\eps>0$, we say that a function $\vphi'$ is $\eps$-strongly close to $\vphi$ and denote this by $\vphi' \approx_\eps \vphi$ if 
$$\text{ for \ any \ } \theta \notin (-\tau,\tau) : \frac{1}{1+\eps} \vphi(\theta)\leq  \vphi'(\theta)\leq (1+\eps) \vphi(\theta).$$ 
\end{definition}

\begin{definition}[Weak closeness of probability measures on the circle] We equip the set of positive measures on the circle with the topology of weak convergence and denote by $d$ the distance that defines it. We denote $\sigma \sim_\rho \sigma'$ when $d(\sigma,\sigma')<\rho$.  
\end{definition}

For an absolutely continuous measure on the circle of the form  $\vphi(\theta) d\theta$, we often  abuse notations and denote the measure simply as $\vphi$. In consequence, we often will use the notation $\sigma \sim_\rho \vphi$ when the measure $\sigma$ is $\rho$-close in the weak topology to the measure $\vphi(\theta) d\theta$.

For $f\in L^2_0(X,\mu)$ we denote by $H(f)$ the cyclic space generated by the family $\{f\circ T^k\}_{k\in \Z}$.


\begin{lemma} \label{lemma.conv} Given $\tau_n\to 0$, $\a_n>0$, $\eps_n<\min(2^{-n},\a_n/2)$ and $\rho_n \to 0$, and a sequence of measures on the circle $\sigma_n$  such that 
$$\sigma_n \sim_{\rho_n} \vphi_n$$
where $\vphi_n$ is a sequence of $(\a_n,\tau_n)$-good functions such that 
$$\varphi_{n+1} \approx_{\eps_n} \vphi_n$$
then $\vphi_n$ converges on $\T\setminus\{0\}$ to a strictly positive continuous function $\vphi_\infty$ and the measures $\sigma_n$ converge weakly to the measure $\sigma$ with density $\vphi_\infty$.
\end{lemma} 

{\begin{proof} From $\sigma_n \sim_{\rho_n} \vphi_n$, it suffices to see that for any $\tau>0$, $\vphi_n$ converges in the strong topology on $[\tau,1-\tau]$ to a strictly positive continuous function $\varphi_\infty$. But the sequence ${\vphi_n}_{|[\tau,1-\tau]}$ is a Cauchy sequence from the fact that $\tau_n\to 0$ and from the fact that $\eps_n<\min(2^{-n},\a_n/2)$.\end{proof}}

 \begin{lemma} \label{lem.trou} Let $(T,A,\mu)$ be a dynamical system such that there exists a sequence $\eps_n\to 0$ and 
 a sequence of 
  functions $\{f_j\}_{j\in \N}$ in $L^2(A,\mu)$, and a sequence of 
 functions $\{\vphi_j\}_{j\in \N}$ in $C^0(\T,\R_+^*)$ such that for every $i \in \N$, 
 \begin{equation}\label{cond1}\sigma(f_i)=\vphi_i \approx_{\eps_i} 1;\end{equation}and for all  $1\leq i<j$, and for $\eta \in \{1,i\}$, there exists $\vphi_{i,j,\eta} \in C^0(\T,\R_+^*)$
 such that  \begin{equation} \label{cond2} 
\bbbbb \sigma(f_i+\eta f_j) = \vphi_{i,j,\eta} \approx_{\eps_i} 2 \end{equation}
then, the system $(T,A,\mu)$ has a spectral component that is Lebesgue  with infinite multiplicity. 
 \end{lemma} 
 
 \begin{proof} Condition \eqref{cond1}
 implies that  the system $(T,A,\mu)$ has a spectral component for which the maximal spectral type is Lebesgue. 
 Let $\bigoplus_{1}^\infty L^{2}(\mathbb{T},\mu_k)$
 be the spectral decomposition of the latter component where $\mu_1=d\theta$, $\mu_1\gg \mu_2 \gg \ldots$. We can take $\mu_k=\chi_{C_k}(\theta)d\theta$ where $\{C_k\}$ is a sequence of nested measurable subsets of the circle. 

If the multiplicity of the Lebesgue component is not infinite, there exists $K$ such that ${\rm Leb}(C_K)<1$. We assume that this holds, and take the first $K\geq 2$ with this property, and we will get a contradiction. 
 
Let $\hat{f}_l=f_{N+l}$, $l\in \{1,\ldots,K+1\}$ for some $N\gg 1$ to be determined later. The goal from choosing $N$ large is to have due to \eqref{cond1} and \eqref{cond2} that the $\hat{f}_l$ are pairwise almost orthogonal while the densities of their spectral measures are almost equal to $1$. This will show that by choosing $N$ sufficiently large, we get a contradiction with the assumption that ${\rm Leb}(C_K)<1$.

 Note that since the spectral measure of every $\hf_l$ is absolutely continuous with respect to Lebesgue, they all spectrally belong to $\bigoplus_{1}^\infty L^{2}(\mathbb{T},\mu_k)$. For every $l\in \{1,\ldots,K+1\}$, let $\hf_{l}^{1},\hf_{l}^{2},\ldots$ denote the successive orthogonal projections of $\hf_{l}$ on $\bigoplus L^{2}(\mathbb{R},\mu_{k})$. 

For any $\eps>0$, if $N$ is chosen sufficiently large,   \eqref{cond1} and \eqref{cond2} imply that for any pair $i\neq j \in \{1,\ldots,K+1\}^2$, and any $n\in \N$, and $\eta \in \{0,1,i\}$
\begin{align} \label{pol1} \left|\langle \hf_i+\eta \hf_j,(\hf_i+\eta \hf_j)\circ T^n \rangle\right|<\eps, \end{align}
which gives  
\begin{equation} \label{spec44} \left|\langle \hf_i,\hf_j\circ T^n \rangle\right|<\eps. \end{equation}
 From the spectral isomorphism this yields
 \begin{equation} \label{spec1} \eps> \left|\langle \hf_i,\hf_j\circ T^n \rangle\right|= \left|\sum_{k=1}^{\infty}\hf_i^{k}(\theta)\overline{{\hf_j^{k}}}(\theta)e(-n\theta) \phi_k(\theta)d\theta\right|. \end{equation}
 
Since $\eps$ can be chosen arbitrarily small (after $K$ and $C_K$ are given), this implies that for more than $99\%$ of $\theta \in C_K^c$, for any pair $i\neq j \in \{1,\ldots,K+1\}^2$,
 \begin{equation} \label{spec2}  \left|\sum_{k=1}^{K}\hf_i^{k}(\theta)\overline{{\hf_j^{k}}}(\theta)\right| <\sqrt{\eps}. \end{equation}
 
 On the other hand by \eqref{cond1} and the spectral isomorphism, we have for any $l \in \{1,\ldots,K+1\}$ 
 \begin{equation} \label{perp} \sum_{k=1}^{\infty} |\hf^{k}_{l}(\theta)|^2\phi_{k}(\theta) =\vphi_i(\theta) \approx_{\eps} 1. \end{equation}
Hence, for more than $99\%$ of $\theta \in C_K^c$, for any $l \in \{1,\ldots,K+1\}$, it holds that 
 \begin{equation}\label{norm}\sum_{k=1}^{K} |\hf^{k}_{l}(\theta)|^2 \in [1-\sqrt{\eps},1+\sqrt{\eps}]. \end{equation}

In conclusion, there exists $\theta \in C_K^c$ for which both \eqref{perp} and \eqref{norm} hold. A contradiction, because $K+1$ vectors in ${\mathbb C}^K$ cannot have all norm almost $1$ and be almost orthogonal.
 \end{proof} 
 \black

\subsection{Simple functions} 

In all the note, we will only consider functions $f$ on the space $(X,\mu)$ or the inductions spaces $(A,\mu_{A})$ that are simple in the sense that $f$ is  constant on the atoms of a finite measurable partition $\mathcal P$ of $A$, and that the average of $f$ is $0$. We will denote this by $f \in \simmple(A)$. 
{ This is useful to make sure that when inducing the function $f$ on a subset $A'\subset A$ as $f':=f_{|A}$, then we can guarantee that ${av}_{\mu_A}(f')=0$ provided $A'$ is independent of the partition $\mathcal P$ that defines $f$. }

\subsection{Criterion for the convergence of a sequence of induced systems to a system with an infinite Lebesgue component.}

We always assume given an ergodic dynamical system $(T,X,\mu)$. We will use the following criterion that allows to construct a measurable set $A \subset X$ such that the system $(T_{A},A,\mu_{A})$ has a Lebesgue component of infinite multiplicity in its spectrum.

\begin{proposition} \label{lemma.convergence}  Assume given  $\tau_n\to 0$, $\a_n>0$, $\eps_n<\min(2^{-n},\a_n/2)$ and $\rho_n \to 0$, and a nested sequence of measurable sets $A_n$  such that $\mu(A_{n-1} \setminus A_n)<\eps_{n}$. 

Assume also given  arrays of functions in $C^0(\T,\R_+)$, $\{\vphi_j^{(n)}\}_{j\in [1,n]}$, $\{\vphi_{i,j,\eta}^{(n)}\}_{1\leq i<j\leq n,\eta \in \{1,i\}}$, $n\geq 1$, such that 
\begin{itemize}
\item[$(A1_n)$] For all $j\in [1,n]$, $\vphi^{(n)}_j$ is $(\a_n,\tau_n)$-good,
\item[$(A2_n)$] For all $j\leq n-1$, for all $1\leq i<j\leq n-1$ and for $\eta \in \{1,i\}$
$$\vphi^{(n)}_j  \approx_{\eps_n} \vphi^{(n-1)}_j, \quad \vphi^{(n)}_{i,j,\eta}  \approx_{\eps_n} \vphi^{(n-1)}_{i,j,\eta}, \quad \vphi^{(n)}_{j,n,\eta}  \approx_{\eps_n} \vphi^{(n)}_j+1, \quad \vphi^{(n)}_n\approx_{\eps_n} 1,$$

\end{itemize} 
and an array of functions $\{f_j^{(n)}\}_{j\in [1,n]} \in \simmple(A_n)$, $n\geq 1$, such that  {$A_n$ is orthogonal to $\{f_j^{(n-1)}\}_{j\in [1,n-1]}$, and } for every $j\in [1,n-1]$
\begin{itemize}
\item[$(A3_n)$]  $\|f_j^{(n)} -{f_j^{(n-1)}}_{|A_n}\|_2 \leq {\eps_n}$;
\item[$(A4_n)$] For the induced system $(T_{|A_n},A_n,\mu_{A_n})$, for all $j\leq n$, for all $1\leq i<j\leq n$, and for $\eta \in \{1,i\}$,
\begin{equation*}
\bbbbb \sigma(f^{(n)}_j) \sim_{\rho_n} \vphi^{(n)}_{j}, \quad \bbbbb \sigma(f^{(n)}_i+\eta f^{(n)}_j) \sim_{\rho_n} \vphi^{(n)}_{i,j,\eta}; \end{equation*}
\end{itemize} 
then, the limiting system $T_{|A_\infty}$ has a spectral component that is Lebesgue  with infinite multiplicity. 
\end{proposition} 

\begin{proof} From the fact that $\mu(A_{n-1} \setminus A_n)<\eps_{n}$ and 
 $(A3_n)$ we see that the system $(T_{|A_\infty},A_\infty,\mu_{A_\infty})$ is well defined and that for every $i$,  
$f^{(n)}_i$ converges in $L^2_0(A_\infty)$ to a function $f^{(\infty)}_i$. 

By $(A1_n), (A2_n), (A4_n)$ with $\eta=0$, Lemma \ref{lemma.conv} implies that for every $j \in \N$, $\sigma(f^{(\infty)}_j)=\vphi^{(\infty)}_j$ is equivalent to the Lebesgue measure on the circle, and $\vphi^{(\infty)}_j  \approx_{\eps_j} 1$. 

Finally $(A4_n)$ and $(A2_n)$ with $\eta\in \{1,i\}$ implies that for $i<j$ \begin{equation*}
\bbbbb \sigma(f^{(\infty)}_i+\eta f^{(\infty)}_j) = \vphi^{(\infty)}_{i,j,\eta} \approx_{\eps_j} \vphi^{(j)}_{i,j,\eta} \approx_{2\eps_j} \vphi^{(j)}_{i}+1\approx_{3\eps_j} \vphi^{(\infty)}_{i}+1 \approx_{3\eps_i} 2 \end{equation*}
 which by Lemma \ref{lem.trou} implies that $(T_{|A_\infty},A_\infty,\mu_{A_\infty})$ has  a spectral component that is Lebesgue  with infinite multiplicity.  
 \end{proof}

\subsection{De La Rue's strategy for inducing Lebesgue spectrum}

In \cite{delarue} the following strategy to induce a Lebesgue spectrum. First of all, one shows how, starting from a simple function $f$, it is possible to induce on a set $A_\infty$ to get a function $f^{(\infty)}$ with spectral measure equivalent to Lebesgue. 

For this purpose $A_\infty$ is constructed inductively as a limit of nested sets $A_n$ and $f^{(\infty)}$ is the limit of the sequence $f^{(n)}:={f^{(n-1)}}_{|A_n}$, $f^{(0)}:=f$. 

Each step of the inductive procedure relies on two mechanisms. First, spread out the spectral measure of $f^{(n)}$ into a measure that is equivalent to Lebesgue using  the Meilijson skew products (\cite{meilijson}) of 
$T$ and $T^2$ above the Bernoulli shift on  $\{1,2\}^\Z$. 
  Second, induce on a set $A_{n+1} \subset A_n$ so that the spectral measure of $f^{(n+1)}= {f^{(n)}}_{|A_{n+1}}$ is as close as desired  in the weak distance to the spread out measure (inducing can imitate Bernoulli convolution). While doing so inductively, it is possible to make sure that the densities of the spread out measures at each step $n$ are converging in the strong sense of Lemma \ref{lemma.conv}, which guarantees  that the spectral measure of $f^{(\infty)}$ for the system $T_{A_\infty}$ is equivalent to Lebesgue. 

Starting from a dense countable family of functions in $\simmple(X,\mu)$ and performing the regularizing induction simultaneously for all the functions, one thus gets a limit system with maximal spectral type equivalent to Lebesgue (see Section \ref{sec.conclusion} where we come back on this density argument since we will need it to conclude our proof of Theorem \ref{main.infinite}). 

The two main ingredients of \cite{delarue} that we just discussed are summarized in the following two propositions that will also be crucial in our proof of Theorem \ref{main.infinite}.

For $f \in L^2_0(X)$ we use the notation $\sigma(f_\delta)$ for the measures $\sigma(f)_\delta$ as in Definition 3 of \cite{delarue}. These are the spreadout measures coming from the Meilijson skew products. For the convenience of the reader we include the definition.

\begin{definition} Given a positive measure $\sigma$ on $\T$ such that $\sigma(\{0\})=0$, and given $\delta>0$, we define a positive measure $\sigma_\delta$ on $\T$ by its Fourier coefficients
\begin{align*}\widehat{\sigma_\delta}(0)&=\sigma(\T) 
\\ \widehat{\sigma_\delta}(p)&=\int_\T z_\delta(\tau)^p d\sigma(\tau), \quad \forall p>0 \\
\widehat{\sigma_\delta}(p)&=\overline{\widehat{\sigma_\delta}({-p})}, \quad \forall p<0
\end{align*}
where $z_\delta(\tau)=(1-\delta)e^{-i\tau}+\delta e^{-2i\tau}$.

Given an ergodic dynamical system 
$(T,X,\mu)$ and  $f \in L^2_0(X)$ and let $\sigma$  be the spectral measure on  $\T$ associated to $f$, then we define  $\sigma(f_\delta):=\sigma_\delta$.

\end{definition} 

\begin{proposition}[Weak closeness implies strong closeness for the spread out densities. {\cite[Lemme 8]{delarue}}] \label{prop.dlr1}
Let $(\vphi_j)$, $j=1,\ldots,K$ be a finite family of $(\a,\tau)$-good functions. For any $\eps>0$, for all  $\delta<\delta(\a,\tau,\eps)$, there exists $\rho>0$ such that if $f_1,\ldots,f_K$ are simple functions such that $\sigma(f_j) \sim_\rho \varphi_j$, then the densities $\vphi_{j,\delta}$ of $\sigma(f_{j,\delta})$ are strictly positive continuous functions on $(0,1)$ and satisfy   $\vphi_{j,\delta} \approx_\epsilon \vphi_j$. 

\end{proposition}

\begin{proposition}[Approaching the spread densities by inducing, {\cite[Proposition 7]{delarue}}] \label{prop.dlr2} If $\delta\in (0,1/2)$, and $f_1,\ldots,f_K$ are simple functions, then for any $\rho>0$ there exists $A$  such that $\mu(A^c)<2\delta$ and for $f'_j={f_j}_{|A}$ and $(T_A,A,\mu_A)$ we have
 $$\sigma(f'_j)  \sim_\rho \sigma(f_{j,\delta}).$$
 Moreover, $A$ can be chosen to  be independent of the functions $f_1,\ldots,f_K$ such that the functions $f'_j$ are simple. 

\end{proposition}



\subsection{Adding one almost orthogonal function to a family}

In our inductive construction to prove  Theorem \ref{main.infinite} based on Lemma \ref{lemma.convergence}, we will need to introduce at each step $n$ an additional function to the almost orthogonal functions already constructed so that at the end we guarantee an infinite multiplicity. For this, we will need the following simple lemma.

The following lemma shows that it is possible to import spectral multiplicity in the weak sense. The regularizing lemma transforms it into an actual increase of the spectral multiplicity.

\begin{lemma}  \label{lemma.add} Let $(X,T)$ be an ergodic dynamical system. Given $\rho > 0$ and a family of $K$ functions $ \{f_j \}_{ j=1,2,...K}$ $\in L^2(X)$, there exists a simple function $f_{k+1}$ such that :
\begin{itemize} 
\item[(a)]  $\sigma (f_{K+1})  \sim_{\rho} 1$
\item[(b)]  for all $\eta \in \{-1, +1\}$, for all $j= 1,2, ...K$,
$$\sigma (f_{K+1} + \eta f_j) \sim_{\rho} \sigma(f_j) +1$$
\end{itemize}
\end{lemma}

\begin{proof} We consider $(Y,B)$ where $B$ is a Bernoulli shift acting on $Y$ and we form the product  $(X\times Y, T\times B)$. Because $B$ is Bernoulli, there exists an $Y$-measurable function $\phi $     (in $L^2(X \times Y)$)
such that  (1) the $(T\times B)^{i}\phi,       i \in Z$,   are orthogonal and furthermore (2) $H_{\phi}$ is orthogonal to $L^2 (X)$ (in $L^2(X \times Y)$). By a simple application of Rokhlin's lemma, we know that if two finite partitions   $P$ and $Q$ are given in 
$(X\times Y)$  where $P$ is $X$-measurable together with an integer $n$ and $\delta>0$, there exists a partition $\tilde{Q}$ which is $X$-measurable such that the two finite partitions  $\overset{n} { \underset{0} {  \bigvee}}  (T\times B)^{i} (P\bigvee Q))$ and $\overset{n} { \underset{0} {  \bigvee}}  (T\times B)^{i} (P\bigvee  \tilde{Q})$ have very close distributions (how close controlled by $\delta$). Applying this to suitable simple functions approximating $\phi$ and the $f_i$'s ($Q$ being the support of $f_{K+1}$ which approximates $\phi$
and $P$ being spanned by the supports of the simple approximations of the $f_i$'s) one gets  (a) and (b) as a direct consequence of (1) and (2). 
\end{proof}

\section{Inducing a system with infinite Lebesgue spectral component}

In this section, we see how we can find  an induced system of any ergodic system $(T,X,\mu)$ that has a spectral component that is Lebesgue  with infinite multiplicity. The proof of the main Theorem \ref{main.infinite}, that we postpone to the last section, will be a direct combination of the construction that we will propose in this section and of the construction in \cite{delarue} of an induced system with pure Lebesgue spectrum. 

\begin{theorem}
\label{main.infinite.component} 
For any ergodic dynamical system $(T,X,\mu)$ there exists an induced system that has a spectral component of pure Lebesgue  type and infinite multiplicity.
\end{theorem}

\subsection{The inductive step} The proof of Theorem \ref{main.infinite.component} relies on the criterion of Proposition \ref{lemma.convergence} and of the following main inductive step that is based
on  Propositions \ref{prop.dlr1} and  \ref{prop.dlr2} and on Lemma \ref{lemma.add}. 

\begin{proposition} \label{prop.step} Suppose $K\in \N$ and $\left\{\vphi_j\right\}_{j=1,\ldots,K}$ are $(\a,\tau)$-good and $\{\vphi_{i,j,\eta}\}_{1\leq i<j\leq K}$ are in  $C^0(\T,\R_+)$.
 For any $\eps>0$, there exist $\rho=\rho(\left\{\vphi_j\right\}_{j=1,\ldots,K},\{\vphi_{i,j,\eta}\}_{1\leq i<j\leq K},\eps)>0$ such that if $\left\{f_j\right\}_{j=1,\ldots,K}$ are simple functions  such that $\forall (i,j,\eta) \in \{1,\ldots,K\}^2\bbbbb \times \{1,i\}$ with $i<j$
\begin{equation} \label{eq.rho1}\bbbbb 
\bbbbb \sigma(f_i) \sim_{\rho} \vphi_{i}, \quad \bbbbb \sigma(f_i+\eta f_j) \sim_{\rho}\vphi_{i,j,\eta}.\end{equation}\bbbbb 
then there exist $\left\{\vphi'_j\right\}_{j=1,\ldots,K}$ such that $\vphi'_j>0$ on $(0,1)$ for every $j$ and  $\{\vphi_{i,j,\eta}\}_{1\leq i<j\leq K}$ in  $C^0(\T,\R_+)$ such that 
$$\vphi'_j \approx_\eps \vphi_j, \quad \vphi'_{i,j,\eta} \approx_\eps \vphi_{i,j,\eta}$$
and such that for any $\rho'>0$, one can find $A \subset X$ such that $\mu(X \setminus A)<\eps$, $A$ orthogonal to $\left\{f_j\right\}_{j=1,\ldots,K}$, and simple functions $f'_j$ defined on $A$ such that $f'_j={f_j}_{|A}$ and for the system $(T_{|A},A,\mu_A)$ we have \begin{equation} \label{eq.rho2}\bbbbb 
\bbbbb \sigma(f'_i) \sim_{\rho'} \vphi'_{i}, \quad \bbbbb \sigma(f'_i+\eta f'_j) \sim_{\rho'}\vphi'_{i,j,\eta}.\end{equation}\bbbbb 
\end{proposition}

Note that since $\vphi'_j>0$ on $(0,1)$, there exists $\a'>0$ such that $\vphi'_j$ is $(\a',\tau/2)$-good, which will allow to iterate the above proposition. 

\subsection{Proof of Theorem \ref{main.infinite.component}} \label{sec.induction}\bbbbb Before proving the proposition, we see how it implies Theorem \ref{main.infinite.component}.
We fix $\tau_n=1/2^{n}$.

From Proposition \ref{lemma.convergence}, it suffices to construct inductively $\a_n>0$, $\eps_n<\min(2^{-n},\a_n/2)$ and $\rho_n \to 0$, and a nested sequence of measurable sets $A_n$ and an array of simple functions $\{f_j^{(n)}\}_{j\in [1,n]} \in \simmple(A_{n})$, $n\geq 1$
such that  $\mu(A_{n-1} \setminus A_n)<\eps_{n}$, $A_n$ orthogonal to $\left\{f_j^{(n-1)}\right\}_{j=1,\ldots,n-1}$, and  arrays of functions $\{\vphi_j^{(n)}\}_{j\in [1,n]}$,  $\{\vphi_{i,j,\eta}^{(n)}\}_{1\leq i<j\leq n\in [1,n]}$, $n\geq 1$, such that $(A1_n)-(A4_n)$ hold.

 Moreover, we suppose that 
 in the construction  $\rho_k$ is given by $$\rho_k=  \rho(\{\vphi^{(k)}_j\}_{j\in [1,k+1]},\{\vphi_{i,j,\eta}\}_{1\leq i<j\leq k+1},\eps_{k+1})$$ where $\rho_k$ is the function from the first part of Proposition \ref{prop.step} and where we took $\vphi^{(k)}_{k+1}:=1$ and $\vphi_{i,k+1,\eta}:=1+\vphi^{(k)}_i$ for $i\in [1,k]$.
 
 \medskip 
\noindent{\bf The construction for $n=1$.} We let $\rho_1=\rho(\{\vphi^{(1)}_1\},\eps_{2})$ where $\rho(\cdot)$ is  given by  the first part of Proposition \ref{prop.step}.

Using Lemma \ref{lemma.add}, we start with $A_1=X$, $\vphi^{(1)}_1 \equiv 1$, and $f_1^{(1)}$ simple and such that
$$\sigma(f_1^{(1)}) \sim_{\rho_1} \vphi^{(1)}_1.$$

 \medskip 
\noindent{\bf Inducing from $n$ to $n+1$.}  Next, we suppose that everything is constructed up to $n$, that is: $A_{n}$ such that $\mu(A_{n-1} \setminus A_{n})<\eps_{n}$ and an array of simple functions $\{f_j^{(n)}\}_{j\in [1,n]} \in \simmple(A_{n})$, $A_n$ orthogonal to $\left\{f_j^{(n-1)}\right\}_{j=1,\ldots,n-1}$,
and  $\{\vphi^{(n)}_j\}_{j\in [1,n]}$ that are $(\a_{n},\tau_{n})$-good for some $\a_{n}>0$, and  $\{\vphi^{(n)}_{i,j,\eta}\}_{1\leq i<j\leq n}$ that satisfy $(A1_{n})-(A4_{n})$.
 We define $\vphi^{(n)}_{i,n+1}=1$ and $\vphi^{(n)}_{i,n+1,\eta}=\vphi^{(n)}_i+1$ and take $\rho_{n+1}=\rho(\{\vphi^{(n)}_j\}_{j\in [1,n+1]},\{\vphi^{(n)}_{i,j,\eta}\}_{1\leq i<j\leq n+1},\eps_{n+1})$.

Using Lemma \ref{lemma.add}, we add to $\{f_j^{(n)}\}_{j\in [1,n]} \in \simmple(A_n)$ a function $f_{n+1}\in \simmple(A_n)$ such that 
\begin{equation} \label{eq.rho3}\bbbbb 
\bbbbb \sigma(f_{n+1}) \sim_{\rho_{n+1}} 1, \quad \bbbbb \sigma(f^{(n)}_i\pm f_{n+1}) \sim_{\rho_{n+1}}\vphi^{(n)}_{i}+1.\end{equation}


Now we apply Proposition \ref{prop.step} to $\{f_1^{(n)},\ldots,f_n^{(n)},f_{n+1}\}$ and $\{\vphi^{(n)}_j\}_{j\in [1,n+1]}$ and  $\{\vphi^{(n)}_{i,j,\eta}\}_{1\leq i<j\leq n+1}$.  Then, the proposition gives us  $A_{n+1}$  orthogonal to $\{f_1^{(n)},\ldots,f_n^{(n)},f_{n+1}\}$ such that $\mu(A_{n} \setminus A_{n+1})<\eps_{n+1}$ and an array of simple functions $\{f_j^{(n+1)}\}_{j\in [1,n+1]} \in \simmple(A_{n+1})$
and  $\{\vphi^{(n+1)}_j\}_{j\in [1,n+1]}$ that are $(\a_{n+1},\tau_{n+1})$-good for some $\a_{n+1}>0$, and  functions $\{\vphi^{(n+1)}_{i,j,\eta}\}_{1\leq i<j\leq n+1} \in C^0(\T,\R_+)$ that satisfy $(A1_{n+1})-(A4_{n+1})$.

In conclusion, Theorem \ref{main.infinite} now follows from Proposition \ref{lemma.convergence}. \hfill $\Box$

\subsection{Proof of Proposition \ref{prop.step}} The proof of Proposition \ref{prop.step} has  two steps, analogous to the main two steps of \cite{delarue}. 

\medskip 

\noindent {\bf Step 1. Spreading out.} \bbbbb  In the first step we elaborate on Proposition \ref{prop.dlr1} and get the following

\begin{lemma} \label{lemma.step1} Suppose $K\in \N$ and $\left\{\vphi_j\right\}_{j=1,\ldots,K}$ are $(\a,\tau)$-good and $\{\vphi_{i,j,\eta}\}_{1\leq i<j\leq K}$ are in  $C^0(\T,\R_+)$.
 For any $\eps>0$, there exist $\rho(\left\{\vphi_j\right\}_{j=1,\ldots,K},\{\vphi_{i,j,\eta}\}_{1\leq i<j\leq K},\eps)>0$ such that if $\left\{f_j\right\}_{j=1,\ldots,K}$ are simple functions  such that $\forall (i,j,\eta) \in \{1,\ldots,K\}^2\bbbbb \times \{1,i\}$
\begin{equation} \label{eq.rho1}\bbbbb 
\bbbbb \sigma(f_i) \sim_{\rho} \vphi_{i}, \quad \bbbbb \sigma(f_i+\eta f_j) \sim_{\rho}\vphi_{i,j,\eta}.\end{equation}\bbbbb 
then for any $\delta$ sufficiently small, we have that $\sigma(f_{j,\delta})$ and $\sigma(f_{i,\delta}+\eta f_{j,\delta})$ have  continuous densities  $\vphi_{j,\delta}$ and  $\varphi_{i,j,\eta,\delta}$, with $\vphi_{j,\delta}>0$ and for $(i,j)\in [1,K]^2$ and $i<j$,
 \begin{equation} \label{eq.rho222}\bbbbb  \sigma(f_{j,\delta})\approx_\eps \varphi_j, \quad   \sigma(f_{i,\delta}+\eta f_{j,\delta})\approx_\eps \varphi_{i,j,\eta}. \end{equation}

\end{lemma}

\begin{proof} The proof is a direct application of Proposition \ref{prop.dlr1} to several functions at the same time.\end{proof}
\medskip 

\noindent {\bf Step 2. Approaching by induction.} Suppose 
 $K\in \N$, $\a,\tau,\eps,\delta>0$, $\left\{\vphi_j\right\}_{j=1,\ldots,K}$, $\{\vphi_{i,j,\eta}\}_{1\leq i<j\leq K}$, $\rho$ and $\left\{f_j\right\}_{j=1,\ldots,K}$ are as in Lemma \ref{lemma.step1}.

We can apply Proposition \ref{prop.dlr2} to the family of simple functions $\left\{f_{j} \right\}_{j=1,\ldots,K}$, and  get for any $\rho'>0$ a set $A$ orthogonal to $\left\{f_{j} \right\}_{j=1,\ldots,K}$ such that $\mu(X\setminus A)<\eps$, and for the simple functions
$f'_j={f_j}_{|A}$ and the system $(T_A,A,\mu_A)$ it holds that  
 \begin{equation} \label{eq.rho2}\bbbbb 
\bbbbb \sigma(f'_i) \sim_{\rho'} \vphi'_{i}:= \sigma(f_{j,\delta}) \approx_\eps \vphi_j, \quad \bbbbb \sigma(f'_i+\eta f'_j) \sim_{\rho'}\vphi'_{i,j,\eta}:= \sigma(f_{i,\delta}+\eta f_{j,\delta}) \approx_\eps \vphi_{i,j,\eta},\end{equation}\bbbbb 
with $\vphi'_j>0$ on $(0,1)$ for every $j$ and  $\{\vphi'_{i,j,\eta}\}_{1\leq i<j\leq K}$ in  $C^0(\T,\R_+)$.

The proof of Proposition \ref{prop.step}  is thus completed. \hfill $\Box$

  \section{Proof of Theorem \ref{main.infinite}} \label{sec.conclusion}
  
To go from Theorem \ref{main.infinite.component} to Theorem \ref{main.infinite}, we can keep the inductive construction of Theorem \ref{main.infinite.component} essentially as is, and add a feature to guarantee that the maximal spectral type is equivalent to Lebesgue. To do so, we just need to make sure that the family of functions we are constructing becomes dense in $L^2_0(X,\mu)$. In fact, we find it simpler to just adjunct to the array 
  of simple functions $\{f_j^{(n)}\}_{j\in [1,n]} \in \simmple(A_n)$ another array of simple functions $\{h_j^{(n)}\}_{j\in [1,n]} \in \simmple(A_n)$ whose role is to guarantee pure Lebesgue spectrum for the final induced system. 
    For this, we follow verbatim \cite[Section 3.3]{delarue}. We recall first the approach in \cite{delarue} to guarantee a pure Lebesgue spectrum for the final induced system:
  
  Start with a family of simple functions $\{h_j\}_{j\in \N}$ that is dense   in $L^2_0(X,\mu)$. At each step of the  construction, pick a set $A_n$ that works simultaneously for the family $\{h_j^{(n)}\}_{j\in [1,n]}$, where $h_j^{(k)}={h_j}^{(k-1)}_{|A_k}$ for every $j\leq k-1$ and $h_k^{(k)}={h_k}_{|A_k}$, in the sense that for every fixed $j$, $\sigma(h_j^{(n)})\sim_{\rho_n} \psi_{j,n}$ where the spectral measures are considered with respect to the induced system $T_{|A_n}$, and   $\psi_{j,n}$ is a sequence of densities that converge in $L^2$ in the sense of Lemma \ref{lemma.conv}. To keep the functions simple at each step of the induction, the set $A_n$ is chosen independent of the partitions that define the simple functions $\{h_j^{(n)}\}_{j\in [1,n]}$.

  At the end of the construction the spectral measure of $h_j^\infty$ for the system $T_{|A_\infty}$ is equivalent to Lebesgue for every $j\in \N$.  
  
 Since for every $j$, $h_j^\infty={h_j}_{|A_\infty}$, it follows from the density of the family $\{h_j\}_{j\in \N}$ that the family  $\{h_j^\infty\}_{j\in \N}$ is dense. Hence the system $(T_{A_\infty},A_\infty)$ has pure Lebesgue spectrum. 
  

  
   Returning to our construction and suppose given at the beginning  
    a family of simple functions $\{h_j\}_{j\in \N}$ that is dense   in $L^2_0(X,\mu)$.

  As in \cite{delarue}, we remark that when we carry on the inductive construction in the proof of Theorem \ref{main.infinite.component}, it is possible to choose $A_{n+1}$ that works simultaneously for $\{f_j^{(n)}\}_{j\in [1,n]}$ (defined as in \S \ref{sec.induction})  as well as for  $\{h_j^{(n)}\}_{j\in [1,n]}$ (defined as above). Note that, in this procedure, at each step $n$ of the induction, the family $\{f_j^{(n)}\}$ and the family  $\{h_j^{(n)}\}_{j\in [1,n]}$
  are updated by mere induction $h_j^{(n)}={h_j}^{(n-1)}_{|A_n}$. Hence,  the density of the family $\{h_j\}_{j\in \N}$ is automatically transmitted to $\{h_j^\infty\}_{j\in \N}$.
 
In conclusion, from the family $\{f_j^{(\infty)}\}_{j\in \N}$ we get the infinite multiplicity of the Lebesgue component, and from the family 
   $\{h_j^{(\infty)}\}_{j\in \N}$ we get that the spectrum is pure Lebesgue. The proof of Theorem \ref{main.infinite} is thus complete. \hfill $\Box$

\end{document}